\documentclass[11pt,reqno,oneside]{amsart}

\usepackage[asymmetric,top=3.5cm,bottom=4.3cm,left=3.1cm,right=3.1cm]{geometry}
\usepackage{amsmath,amsfonts,amsthm,mathrsfs,amssymb,cite}
\usepackage[usenames]{color}

\usepackage{mathtools} 
\usepackage{graphics}
\usepackage{framed}

\usepackage[colorlinks,citecolor=cyan,linkcolor=cyan]{hyperref} 

\usepackage{enumerate}

\usepackage{graphicx}
\usepackage{latexsym}
\usepackage{enumerate}

\theoremstyle{plain}
\newtheorem{theorem}{Theorem}[section]

\newtheorem{lemma}{Lemma}[section]

\newtheorem{definition}{Definition}[section]

\newtheorem{remark}{Remark}[section]

\renewcommand{\eqref}[1]{\textnormal{(\ref{#1})}}

\numberwithin{equation}{section}

\begin{document}

\title[Surface-localized transmission eigenmodes]{On new surface-localized transmission eigenmodes}

\author{Youjun Deng}
\address{School of Mathematics and Statistics, HNP-LAMA, Central South University, Changsha, Hunan, China.}
\email{youjundeng@csu.edu.cn, dengyijun\_001@163.com}

\author{Yan Jiang}
\address{Department of Mathematics, Jilin University, Changchun, Jilin, China.}
\email{jiangyan20@mails.jlu.edu.cn}

\author{Hongyu Liu}
\address{Department of Mathematics, City University of Hong Kong, Hong Kong SAR, China.}
\email{hongyu.liuip@gmail.com, hongyliu@cityu.edu.hk}

\author{Kai Zhang}
\address{Department of Mathematics, Jilin University, Changchun, Jilin, China.}
\email{zhangkaimath@jlu.edu.cn}


\begin{abstract}

Consider the transmission eigenvalue problem
\[
(\Delta+k^2\mathbf{n}^2) w=0,\ \ (\Delta+k^2)v=0\ \ \mbox{in}\ \ \Omega;\quad w=v,\ \ \partial_\nu w=\partial_\nu v=0\ \ \mbox{on} \ \partial\Omega.
\]
It is shown in \cite{CDHLW} that there exists a sequence of eigenfunctions $(w_m, v_m)_{m\in\mathbb{N}}$ associated with $k_m\rightarrow \infty$ such that either $\{w_m\}_{m\in\mathbb{N}}$ or $\{v_m\}_{m\in\mathbb{N}}$ are surface-localized, depending on $\mathbf{n}>1$
or $0<\mathbf{n}<1$. In this paper, we discover a new type of surface-localized transmission eigenmodes by constructing a sequence of transmission eigenfunctions $(w_m, v_m)_{m\in\mathbb{N}}$ associated with $k_m\rightarrow \infty$ such that both $\{w_m\}_{m\in\mathbb{N}}$ and $\{v_m\}_{m\in\mathbb{N}}$ are surface-localized, no matter $\mathbf{n}>1$ or $0<\mathbf{n}<1$. Though our study is confined within the radial geometry, the construction is subtle and technical.

\medskip



\end{abstract}

\maketitle

\section{Introduction}

\subsection{Mathematical setup and discussion on the major finding}

Let $\Omega$ be a bounded domain in $\mathbb{R}^N$, $N=2, 3$, with a connected complement $\mathbb{R}^N\backslash\overline{\Omega}$ and $\mathbf{n}\in L^\infty(\Omega)$ be a positive function. Consider the following transmission eigenvalue problem for $w \in H^1(\Omega)$ and $v\in H^1(\Omega)$:
\begin{equation}\label{eq:trans1}
\left\{
\begin{array}{ll}
\Delta w+k^2\mathbf{n}^2 w=0  &\text{in} \ \Omega,\medskip \\
\Delta v+k^2 v =0 &\text{in} \ \Omega, \medskip \\
\displaystyle{w=v,\ \ \frac{\partial w}{\partial\nu}=\frac{\partial v}{\partial\nu} } &\text{on} \ \partial \Omega, \\
\end{array}
\right.
\end{equation}
where $\nu$ is the exterior unit normal vector to $\partial\Omega$. Clearly, $w\equiv v\equiv 0$ are a pair of trivial solutions to \eqref{eq:trans1}. If there exists a non-trivial pair of solutions $(w, v)$ to \eqref{eq:trans1}, $k\in\mathbb{R}_+$ is called a transmission eigenvalue, and $w, v$ are the associated transmission eigenfunctions. The transmission eigenvalue problem connects to the inverse acoustic scattering theory in many aspects in delicate and mysterious manners, from unique identifiability, reconstruction algorithm to invisibility cloaking. Its study has a long and colourful history in the literature, and we refer to \cite{CKreview,CHreview,Liureview} for historical accounts and surveys on the state-of-the-arts developments.

In this paper, we are mainly concerned with the geometry of the transmission eigenfunctions, which was initiated in \cite{BL2017b} by showing that the transmission eigenfunctions generically vanish around a corner. The discovery was inspired by the corresponding study in the context of characterizing the wave scattering from corner singularities \cite{BPS}. The major difference is that the weaker regularities of the transmission eigenfunctions than those in the scattering problems require more technical and subtle treatments. The study has received considerable attentions recently in the literature; see  \cite{Bsource,EBL,BL2016,BLLW,BL2017,BL2018,BXL,CX,CDL,DCL,DDL,CDL,SPV} and the references cited therein. The results mentioned above are all of a local feature which are localized around certain peculiar geometrical points on $\partial\Omega$. In \cite{CDHLW}, a global geometric rigidity property was discovered for the transmission eigenfunctions. In fact, it is shown that there exists a sequence of eigenfunctions $(w_m, v_m)_{m\in\mathbb{N}}$ associated with $k_m\rightarrow \infty$ such that either $\{w_m\}_{m\in\mathbb{N}}$ or $\{v_m\}_{m\in\mathbb{N}}$ are surface-localized, depending on $\mathbf{n}>1$
or $0<\mathbf{n}<1$. In fact, if $\mathbf{n}>1$, $\{v_m\}_{m\in\mathbb{N}}$ are surface-localized, but $\{w_m\}_{m\in\mathbb{N}}$ are not surface-localized; whereas if $0<\mathbf{n}<1$, $\{w_m\}_{m\in\mathbb{N}}$ are surface-localized, but $\{v_m\}_{m\in\mathbb{N}}$ are not surface-localized. Here, by surface-localization, we mean the $L^2$-energy of the transmission eigenmode is concentrated in a sufficiently small neighbourhood of $\partial\Omega$. Moreover, two interesting applications of practical importance were generated by using the surface-localization of the transmission eigenmodes in \cite{CDHLW}, including a super-resolution wave scheme and a possible pseudo plasmon sensing scheme.

In the current article, we show that there exists a new type of surface-localised transmission eigenmodes which are different from those found in \cite{CDHLW}. In fact, we shall construct a sequence of transmission eigenfunctions $(w_m, v_m)_{m\in\mathbb{N}}$ associated with $k_m\rightarrow \infty$ such that both $\{w_m\}_{m\in\mathbb{N}}$ and $\{v_m\}_{m\in\mathbb{N}}$ are surface-localized, no matter $\mathbf{n}>1$ or $0<\mathbf{n}<1$. We shall mainly consider our study for the case that $\Omega$ is radially symmetric and $\mathbf{n}$ is a constant. Though the result is presented for a special setup, it turns out that the construction is subtle and technical. We believe the result holds for more general case, which shall be the subject of our future study.

\subsection{Connection to inverse scattering theory}

In order to provide a physical background of our spectral study, we briefly discuss the time-harmonic acoustic scattering due to an incident field $u^i$ and an inhomogeneous refractive medium $(\Omega, \mathbf{n})$. Here, $u^i$ is an entire solution to homogeneous Helmholtz equation,
\begin{equation}\label{eq:homnn1}
\Delta u^i + k^2 u^i =0\quad \text{in} \  \mathbb{R}^N.
\end{equation}
In the physical context, $k\in\mathbb{R}_+$ is the wavenumber and $\mathbf{n}$ is the refractive index. For notational convenience, we extend $\mathbf{n}$ to be $1$ outside $\Omega$. The forward scattering problem is described by the following Helmholtz system
\begin{equation}\label{eq:forwardpro}
\left\{
\begin{array}{ll}
\Delta u+k^2 \mathbf{n}^2(x)u =0  & \text{in} \  \mathbb{R}^N, \medskip \\
u = u^i+u^s  & \text{in} \  \mathbb{R}^N, \medskip \\
\displaystyle \lim\limits_{r\to\infty} r^{\frac{N-1}{2}}\left(\partial_r-\mathrm{i}k\right) u^s=0, & r=|x|,
\end{array}
\right.
\end{equation}
where $u$ and $u^s$ are respectively referred to as the total and scattered fields. The last limit in \eqref{eq:forwardpro} is known as the Sommerfeld radiation condition which holds uniformly in the angular variable $\hat x:=x/|x|$ and characterizes the outgoing nature of the scattered wave field $u^s$.  The well-posedness of the scattering system \eqref{eq:forwardpro} is known \cite{LSSZ}, and there exists a unique solution $u\in H_{loc}^2(\mathbb{R}^d)$ which admits the following asymptotic expansion:
\[
u^s(x; u^i, (\Omega,\mathbf{n}))=\frac{e^{\mathrm{i}k r}}{r^{(N-1)/2}} u_\infty(\hat x)+\mathcal{O}(r^{-(N+1)/2})\quad \mbox{as}\ r\rightarrow\infty.
\]
$u_\infty$ is known as the far-field pattern. Introduce an abstract operator which sends the inhomogeneity $(\Omega,\mathbf{n})$ to its far-field pattern under the probing of the incident field as follows:
\begin{equation}\label{eq:fo1}
\mathcal{F}((\Omega,\mathbf{n}), u^i)=u_\infty.
\end{equation}
An inverse problem of practical importance is to recover the refractive inhomogeneity $(\Omega,\mathbf{n})$ by knowledge of the far-field measurement. From the inverse problem point of view, it seems more practical for one to characterize the range of $\mathcal{F}$, namely $\mathrm{Range}(\mathcal{F})$ associated with all possible $u^i$, which contains all the ``visible" information. However, a different perspective was proposed in \cite{CDHLW} and in fact one can achieve super-resolution reconstruction for the inverse problem \eqref{eq:fo1} if instead using $\mathrm{Ker}(\mathcal{F})$. Indeed, it is shown in \cite{CDHLW} that $\mathcal{K}_\epsilon(\mathcal{F})$ can be obtained by knowledge of $\mathrm{Range}(\mathcal{F})$, where $\epsilon\in\mathbb{R}_+$ is sufficiently small. Here, $u^i\in \mathcal{K}_\epsilon(\mathcal{F})$ means that $\|\mathcal{F}((\Omega,\mathbf{n}), u^i)\|\leq \epsilon$. It turns out that the $\mathrm{Ker}_\epsilon(\mathcal{F})$ actually consists of the Herglotz extensions of the transmission eigenfunctions in \eqref{eq:trans1}. Hence, the surface-localized transmission eigenmodes carry the geometric information of the underlying refractive inhomogeneity, which forms the basis of the super-resolution imaging scheme in \cite{CDHLW}. Therefore, our study not only unveils a new spectral phenomenon of high theoretical value, but also is of significant practical interest.


\section{Main results}

Let us consider the transmission eigenvalue problem \eqref{eq:trans1} with $\Omega$ being a ball in $\mathbb{R}^N$, $N=2, 3$, and $\mathbf{n}$ being a positive constant. By scaling and translation if necessary, we can assume that $\Omega$ is the unit ball, namely $\Omega:=\{x\in\mathbb{R}^N; |x|<1\}$. In what follows, we set
\begin{equation}\label{eq:tau1}
\Omega_\tau:=\{x\in\mathbb{R}^N; |x|<\tau\},\quad \tau\in (0, 1).
\end{equation}

\begin{definition}\label{def:1}
Consider a function $\psi\in L^2(\Omega)$. It is said to be surface-localized if there exists $\tau_0\in (0, 1)$, sufficiently close to $1$, such that
\begin{equation}\label{eq:localized1}
\frac{\|\psi\|_{L^2(\Omega_{\tau_0})}}{\|\psi\|_{L^2(\Omega)}}\ll 1.
\end{equation}
\end{definition}
It is easy to see that if $\psi$ is surface-localized, then its $L^2$-energy concentrate in a small neighbourhood of $\partial\Omega$, namely $\Omega\backslash\Omega_{\tau_0}$. The qualitative asymptotic smallness in \eqref{eq:localized1} shall become more quantitatively definite in what follows. In fact, we can prove
\begin{theorem}\label{thm:main1}
Consider the transmission eigenvalue problem \eqref{eq:trans1} and assume that $\Omega$ is the unit ball and $\mathbf{n}\neq 1$ is a positive constant. Then for any given $\tau\in (0, 1)$, there exists a sequence of transmission eigenfunctions $\{w_m, v_m\}_{m\in\mathbb{N}}$ associated to eigenvalues $k_m$ such that
\begin{equation}\label{eq:result1}
k_m\rightarrow\infty\ \mbox{as}\ m\rightarrow\infty \ \ \mbox{and}\ \ \lim_{m\rightarrow\infty} \frac{\|\psi_m\|_{L^2(\Omega_\tau)}}{\|\psi_m\|_{L^2(\Omega)}}=0,\ \ \psi_m=w_m, v_m.
\end{equation}
\end{theorem}

\begin{remark}
By \eqref{eq:result1}, it is clear that both $\{w_m\}_{m\in\mathbb{N}}$ and $\{v_m\}_{m\in\mathbb{N}}$ are surface-localized according to Definition~\ref{def:1}. In our subsequent proof of Theorem~\ref{eq:result1}, it can be seen that the transmission eigenmode corresponding to higher mode number is more localized around the surface. In other words, in \eqref{eq:result1}, $\tau$ can be very close to $1$ provided $k_m$ is sufficiently large. It is interesting to point out that $1-\tau$ is actually the localizing radius of the eigenmode, which defines the super-resolution power of the wave imaging scheme proposed in \cite{CDHLW}.
\end{remark}

\begin{remark}
It is sufficient for us to prove Theorem~\ref{thm:main1} only for the case $\mathbf{n}>1$. In fact, let us suppose that Theorem~\ref{thm:main1} holds true for $\mathbf{n}>1$, and instead consider the other case with $0<\mathbf{n}<1$. Set $\tilde k=k\mathbf{n}$, $\tilde{\mathbf{n}}=\mathbf{n}^{-1}$, $\tilde w=v$ and $\tilde v=w$. Then \eqref{eq:trans1} can be recast as
\begin{equation}\label{eq:trans11}
\left\{
\begin{array}{ll}
\Delta \tilde w+\tilde{k}^2\tilde{\mathbf{n}}^2 \tilde{w}=0  &\text{in} \ \Omega,\medskip \\
\Delta \tilde{v}+\tilde{k}^2 \tilde{v} =0 &\text{in} \ \Omega, \medskip \\
\displaystyle{\tilde w=\tilde v,\ \ \frac{\partial \tilde{w}}{\partial\nu}=\frac{\partial \tilde{v}}{\partial\nu} } &\text{on} \ \partial \Omega.  \\
\end{array}
\right.
\end{equation}
Since $\tilde{\mathbf{n}}>1$, we readily have that there exist $(v_m, w_m)=(\tilde{w}_m, \tilde{v}_m)$, $m\in\mathbb{N}$, associated to $k_m=\tilde{k}_m/\mathbf{n}\rightarrow\infty$, which are surface-localized. Hence, throughout the rest of the paper, we assume that $\mathbf{n}>1$.
\end{remark}

\subsection{Two-dimensional result}

In this subsection, we prove Theorem~\ref{thm:main1} in the two-dimensional case. Let $x=(r\cos\theta, r\sin\theta)\in\mathbb{R}^2$ denote the polar coordinate.
By Fourier expansion, the solutions to \eqref{eq:trans1} have the following series expansions:
\begin{equation}\label{eq:series1}
w(x)=\sum_{m=0}^\infty \alpha_m J_m(k\mathbf{n}|x|) e^{\mathrm{i}m\theta}, \quad v(x)=\sum_{m=0}^\infty \beta_m J_m(k|x|) e^{\mathrm{i}m\theta},
\end{equation}
where $J_m$ is the $m$-th order Bessel function and $\alpha_m, \beta_m\in\mathbb{C}$ are the Fourier coefficients. Set
\begin{equation}\label{eq:series2}
w_m(x)=\alpha_m J_m(k\mathbf{n}|x|) e^{\mathrm{i}m\theta}, \quad v_m(x)=\beta_m J_m(k|x|) e^{\mathrm{i}m\theta}.
\end{equation}
In what follows, we shall construct the surface-localized transmission eigenmodes of the form \eqref{eq:series2} to fulfil the requirement in Theorem~\ref{thm:main1}. In order to make $(w_m, v_m)$ in \eqref{eq:series2} transmission eigenfunctions of \eqref{eq:trans1}, one has by using the two transmission conditions on $\partial\Omega$, together with straightforward calculations that
\begin{equation}\label{eq:d1}
\beta_m=\frac{J_m(k\mathbf{n})}{J_m(k)}\alpha_m,
\end{equation}
and $k$ must be a root of the following function
\begin{equation}\label{eq:d2}
f_m(k)=J_{m-1}(k) J_m(k\mathbf{n})-\mathbf{n} J_m(k) J_{m-1}(k\mathbf{n}),\ \ m\geq 1.
\end{equation}
Next, we prove the existence of transmission eigenvalues by finding roots of $f_m$. In the sequel, we let $j_{m,s}$ denote the $s$-th positive root of $J_m(t)$ (arranged according to the magnitude), and $j_{m,s}'$ denote the $s$-th positive root of $J_m'(t)$. Here, it is pointed out that both $J_m(t)$ and $J_m'(t)$ possess infinitely many positive roots, accumulating only at $\infty$ (cf. \cite{LiuZou}).

\begin{lemma}\label{lem:1}
Let $\mathbf{n}>1$ and $s_0\in\mathbb{N}$ be fixed. Then there exists $m_0(\mathbf{n}, s_0)\in\mathbb{N}$ such that when $m>m_0(\mathbf{n}, s_0)$, one has
\begin{equation}\label{eq:p1}
\frac{j_{m,s_0}}{\mathbf{n}}\leq m.
\end{equation}
\end{lemma}

\begin{proof}
According to the formula (1.2) in \cite{WQ99}, we know
\begin{equation*}
m-\frac{a_{s}}{2^{1 / 3}} m^{1 / 3}<j_{m, s}<m-\frac{a_{s}}{2^{1 /
3}} m^{1 / 3}+\frac{3}{20} a_{s}^{2} \frac{2^{1 / 3}}{m^{1 / 3}},
\end{equation*}
where $a_s$ is the $s$-th negative zero of the Airy function and has the representation
\begin{equation}\label{eq:dd1}
a_{s}=-\left[\frac{3 \pi}{8}(4 s-1)\right]^{2 /
3}(1+\sigma_{s}).
\end{equation}
Here, $\sigma_s$ in \eqref{eq:dd1} can be estimated by
\begin{equation}
0 \leq \sigma_{s} \leq 0.130\left[\frac{3 \pi}{8}(4
s-1.051)\right]^{-2}.
\end{equation}
By combining the above estimates, one can show by straightforward calculations that
\begin{equation*}
\frac{j_{m,s_0}}{\mathbf{n}}\leq m.
\end{equation*}

The proof is complete.
\end{proof}

\begin{lemma}\label{lem:2}
Let $\mathbf{n}>1$ and $s_0\in\mathbb{N}$ be fixed. Then there exists $m_0(\mathbf{n}, s_0)\in\mathbb{N}$ such that when $m>m_0(\mathbf{n}, s_0)$, the function $f_m(k)$ in \eqref{eq:d2} possesses at least one zero point in $(\frac{j_{m,s_0}}{\mathbf{n}}$, $\frac{j_{m, s_0+1}}{\mathbf{n}})$.

\end{lemma}

\begin{proof}
According to the formula (9.5.2) in \cite{Abr}, we know
\begin{equation*}
m \leq j_{m, 1}'<j_{m, 1}<j_{m, 2}'<j_{m, 2}<j_{m,
3}'<\cdots.
\end{equation*}
By Lemma \ref{lem:1}, one has $\frac{j_{m,
s_0+1}}{\mathbf{n}}\leq m\leq j_{m,1}'$ for $m>m_0(\mathbf{n},s_0+1)$. Next we consider $f_{m}(k)$ for $k\in(\frac{j_{m, s_0}}{\mathbf{n}}, \frac{j_{m,
s_0+1}}{\mathbf{n}})$. We have
\begin{equation*}
J_{m}(k) \geq 0, \quad k \in\left[0, j_{m,
1}'\right],
\end{equation*}
and this implies $J_{m}(k)>0$ for $k\in(\frac{j_{m, s_0}}{\mathbf{n}}, \frac{j_{m,
s_0+1}}{\mathbf{n}})$. Following Lemma~2.1 in \cite{LZ18}, we know that the positive zeros of $J_{m-1}$ are interlaced with those of $J_{m}$, and hence
\begin{equation*}
J_{m-1}(j_{m, s_0}) \cdot J_{m-1}(j_{m,
s_0+1})<0.
\end{equation*}
By using the above fact, one can show that
\begin{equation}\label{eq:ddd3}
\begin{aligned}
& f_{m}\left(\frac{j_{m, s_0}}{\mathbf{n}}\right) \cdot f_{m}\left(\frac{j_{m, s_0+1}}{\mathbf{n}}\right)\\
=&\, \mathbf{n}^2 J_{m}\left(\frac{j_{m, s_0}}{\mathbf{n}}\right) \cdot J_{m}\left(\frac{j_{m, s_0+1}}{\mathbf{n}}\right)
\cdot J_{m-1}(j_{m, s_0}) \cdot J_{m-1}(j_{m, s_0+1}) \\
\leq\, & \mathbf{n}^2(J_{m}(j'_{m, 1}))^{2} \cdot J_{m-1}(j_{m, s_0}) \cdot J_{m-1}(j_{m, s_0+1}) \\
<&\, 0.
\end{aligned}
\end{equation}
which readily implies by Rolle's theorem that there exists at least one zero point of $f(k)$ in $(\frac{j_{m,s_0}}{\mathbf{n}}$, $\frac{j_{m, s_0+1}}{\mathbf{n}})$.

The proof is complete.
\end{proof}
Clearly, Lemma~\ref{lem:2} proves the existence of transmission eigenvalues. In what follows, for a fixed $s_0$, we let the transmission eigenvalue be denoted by
\begin{equation}\label{eq:te1}
k_{l_{m}}:=k_{m, s_0} \in \left(\frac{j_{m, s_0}}{\mathbf{n}},
\frac{j_{m, s_0+1}}{\mathbf{n}}\right), \quad m=m_{0}+1, m_{0}+2,
m_{0}+3, \cdots,
\end{equation}
where $m_0=m_0(\mathbf{n}, s_0)$ is sufficiently large fulfilling the requirements in Lemmas~\ref{lem:1} and \ref{lem:2}.

\begin{lemma}\label{lem:3}
There exists constants $C$ amd $\gamma$ such that
\begin{equation}\label{eq:ee1}
    \frac{J'_{m}(k_{l_{m}})}{J_{m}(k_{l_{m}})}\leqslant
    Cm^{\gamma}.
\end{equation}
\end{lemma}
\begin{proof}
By Theorem~1 in \cite{Kra06}, we have for $0<x<\sqrt{(m+1)(m+3)}$ that
\begin{equation}\label{eq:J'quoJ}
\begin{split}
& \frac{J_{m}'(x)}{J_{m}(x)}\\
\leq & \frac{4 x^{2}-12 m-6+\sqrt{((2m+1)(2m+3)-4 x^{2})^{3}+((2m+1)(2m+3))^{2}}}{2 x((2 m+1)(2 m+5)-4 x^{2})}.
\end{split}
\end{equation}
The items in the numerator of the RHS of \eqref{eq:J'quoJ} can be estimated by
\begin{equation}\label{eq:J'quoJ1}
\begin{aligned}
 &4x^2-12 m-6\leqslant 4m^2,\\
 &\sqrt{((2m+1)(2m+3)-4 x^{2})^{3}+((2m+1)(2m+3))^{2}}\leqslant 4m^3.
\end{aligned}
\end{equation}
According to the formula (1.2) in \cite{WQ99}, we further have
\begin{equation*}
m-\frac{a_{s}}{2^{1/3}} m^{1/3}<j_{m, s}<m-\frac{a_{s}}{2^{1/3}} m^{1/3}+\frac{3}{20} a_{s}^{2} \frac{2^{1/3}}{m^{1/3}}.
\end{equation*}
Hence, it holds that
\begin{equation*}
\frac{m}{\mathbf{n}}\left(1+2(\frac{s_{0}}{m})^{\frac{2}{3}}\right)
\leqslant\frac{j_{m,s_0}}{\mathbf{n}}<k_{l_{m}}<\frac{j_{m,s_0+1}}{\mathbf{n}}
\leqslant\frac{m}{\mathbf{n}}\left(1+3(\frac{s_{0}+1}{m})^{\frac{2}{3}}+2(\frac{s_{0}+1}{m})^{\frac{4}{3}}\right).
\end{equation*}
One thus has for sufficiently large $m$ that
\begin{equation*}
\frac{m}{\mathbf{n}}\leqslant k_{l_{m}}\leqslant \frac{(\mathbf{n}+1)m}{2\mathbf{n}}.
\end{equation*}
By substituting $x=k_{l_{m}}$ in \eqref{eq:J'quoJ}, when $\mathbf{n}>1$
\begin{equation}\label{eq:J'quoJ2}
2x((2m+1)(2m+5)-4x^2)>\frac{2m}{\mathbf{n}}(4m^2+12m+5-4m^2)>\frac{1}{\mathbf{n}}.
\end{equation}
Finally, by substituting (\ref{eq:J'quoJ1}) and (\ref{eq:J'quoJ2}) in (\ref{eq:J'quoJ}), together with straightforward calculations, we can arrive at \eqref{eq:ee1}.

The proof is complete.
\end{proof}

We are in a position to present the proof of Theorem~\ref{thm:main1}, which shall be split into two theorems as follows.

\begin{theorem}\label{thm:surloc}
Consider the same setup as that in Theorem~\ref{thm:main1} in $\mathbb{R}^2$ and assume that $\mathbf{n}>1$ and $\tau\in (0, 1)$ is fixed. Let $(w_m, v_m)$ be the pair of transmission eigenfunctions associated with $k_{l_m}$ in \eqref{eq:te1}. Then it holds that
\begin{equation}
\lim\limits_{m\rightarrow\infty}\frac{\|v_m\|_{L^2(\Omega_{\tau})}}{\|v_m\|_{L^2(\Omega)}}=0.
\end{equation}
\end{theorem}

\begin{proof}
Let $\beta_{m}=1$ in \eqref{eq:d1}. Then one has
\begin{equation*}
\begin{aligned}
\left\|v_{m}\right\|_{L^{2}(\Omega_{\tau})}^{2} &=\int_{\Omega_{\tau}}\left|J_{m}(k_{l_{m}}|x|)\right|^{2} \mathrm{d} x \\
&=2 \pi \int_{0}^{\tau}\left|J_{m}(k_{l_{m}} r)\right|^{2} r \mathrm{d} r\\
&=2 \pi \int_{0}^{\tau}rJ^{2}_{m}(k_{l_{m}} r)
\mathrm{d} r.
\end{aligned}
\end{equation*}
Similarly,
\begin{equation*}
\left\|v_{m}\right\|_{L^{2}(\Omega)}^{2} =2 \pi
\int_{0}^{1}rJ^{2}_{m}(k_{l_{m}} r)  \mathrm{d} r\
\text{.}
\end{equation*}
Set $f(r)=rJ_m^2(k_{l_{m}}r)$. By straightforward calculations, we have
\begin{eqnarray*}
f'(r)&=&J_m^2(k_{l_{m}}r)+2k_{l_{m}}rJ_m(k_{l_{m}}r)J_m'(k_{l_{m}}r), \nonumber\\
f''(r)&=&4k_{l_{m}}J_m(k_{l_{m}}r)J_m'(k_{l_{m}}r)+2k_{l_{m}}^2rJ_m^{'2}(k_{l_{m}}r)+2k_{l_{m}}^2rJ_m(k_{l_{m}}r)J_m'(k_{l_{m}}r) \nonumber\\
&=&4k_{l_{m}}J_m(k_{l_{m}}r)J_m'(k_{l_{m}}r)+2k_{l_{m}}^2rJ_m^{'2}(k_{l_{m}}r)\\
&+&2k_{l_{m}}^2rJ_m(k_{l_{m}}r)\left((\frac{m^2}{k_{l_{m}}^2r^2}-1)J_{m}(k_{l_{m}}r)-\frac{1}{k_{l_{m}}r}J_m'(k_{l_{m}}r)\right)\nonumber\\
&=&2k_{l_{m}}J_m(k_{l_{m}}r)J_m'(k_{l_{m}}r)+2k_{l_{m}}^2rJ_m^{'2}(k_{l_{m}}r)+2k_{l_{m}}^2rJ_m(k_{l_{m}}r)
\left(\frac{m^2}{k_{l_{m}}^2r^2}-1\right)J_{m}(k_{l_{m}}r). \nonumber
\end{eqnarray*}
Noting that $k_{l_{m}}r\leqslant m$, one clearly has
\begin{equation}\label{eq:ff1}
\frac{m^2}{k_{l_{m}}^2r^2}-1\geq0,\
f^{''}(r)\geq0,\quad r\in[0,1].
\end{equation}
Hence $f(r)$ is a convex function on $[0,1]$. Therefore, $\int_{0}^{1}f(r)dr$ is bigger than the area of the triangle under the tangent of $f(1)$ (see \ref{fig:1} for a schematic illustration).
\begin{figure}[h]
\centering
\includegraphics[width=0.4\textwidth]{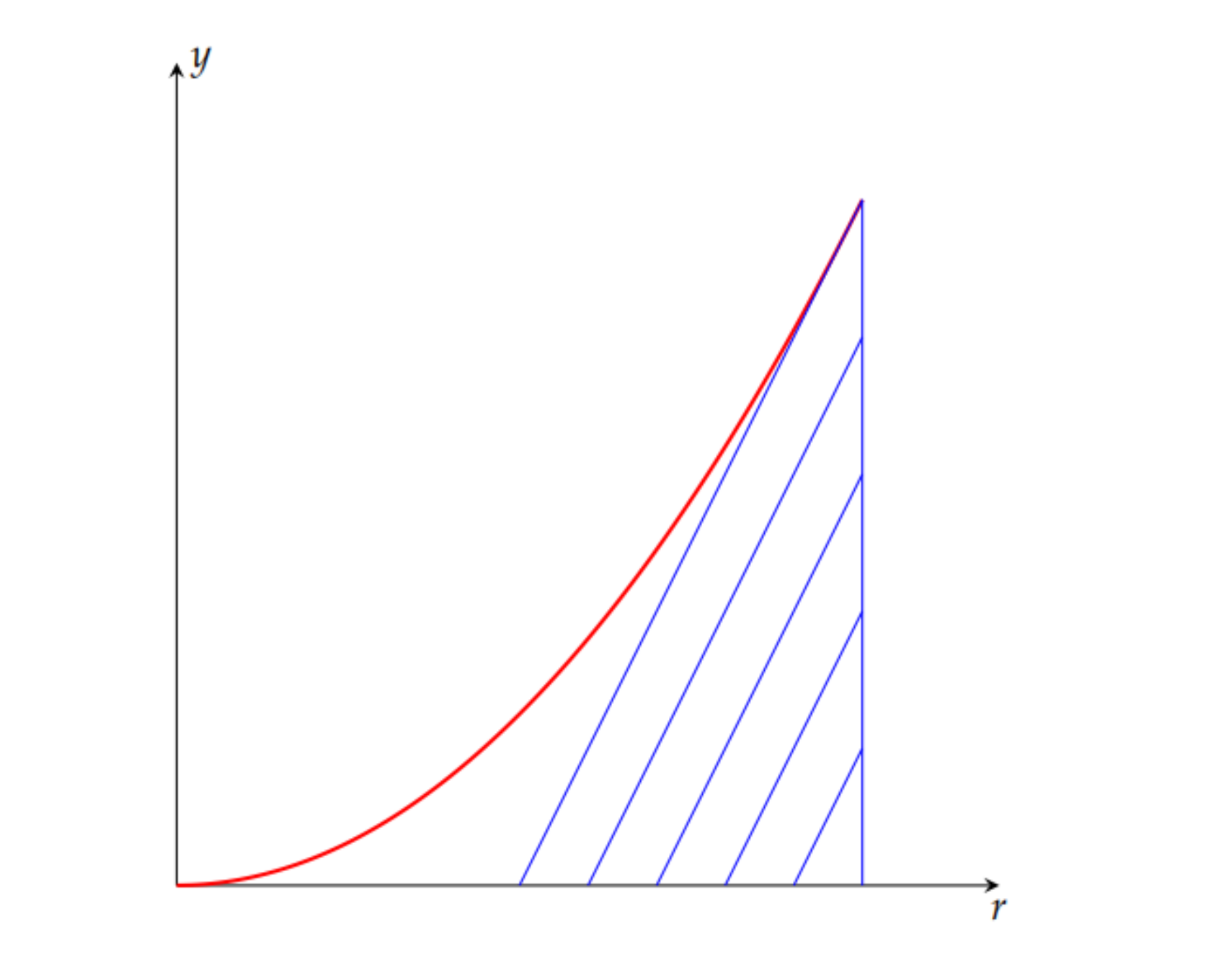}
\caption{\label{fig:1} Schematic illustration of
$\int_{0}^{1}f(r)dr$ which is bigger than the area of the triangle
under the tangent of $f(1)$. }
\end{figure}
The hypotenuse of the aforementioned triangle is
\begin{equation*}
p=(J^2_{m}(k_{l_{m}})+2k_{l_{m}}J_{m}(k_{l_{m}})J'_{m}(k_{l_{m}}))(r-1)+J^2_{m}(k_{l_{m}}),
\end{equation*}
and the lengths of its base and height are respectively $l:=\frac{J_{m}(k_{l_{m}})}{J_{m}(k_{l_{m}})+2k_{l_{m}}J'_{m}(k_{l_{m}})}
$ and $h:=J^2_{m}(k_{l_{m}})$. Hence, the area of the triangle is
\begin{equation*}
\begin{aligned}
S_{\triangle}=\frac{1}{2}lh&=\frac{\frac{1}{2}J^3_{m}(k_{l_{m}})}{J_{m}(k_{l_{m}})+2k_{l_{m}}J'_{m}(k_{l_{m}})}\\
&\leqslant\int_{0}^{1}rJ^2_{m}(k_{l_{m}}r)\mathrm{d}r.
\end{aligned}
\end{equation*}
Since $f$ is monotonically increasing, one has $\int_{0}^{\tau}f(r)\mathrm{d}r\leqslant\tau
f(\tau)$,
\begin{equation*}
\int_{0}^{\tau}rJ^2_{m}(k_{l_{m}}r)\mathrm{d}r\leqslant
\tau^2J^2_{m}(k_{l_{m}}\tau).
\end{equation*}
Therefore, it holds that
\begin{equation*}
\begin{aligned}
\frac{\|v_m\|^2_{L^2(\Omega_{\tau})}}{\|v_m\|^2_{L^2(\Omega)}}&=\frac{\int_{0}^{\tau}rJ^2_{m}(k_{l_{m}}r)\mathrm{d}r}
{\int_{0}^{1}rJ^2_{m}(k_{l_{m}}r)\mathrm{d}r}\\
&\leqslant \displaystyle \frac{\tau^2J^2_{m}(k_{l_{m}}\tau)}{\frac{\frac{1}{2}J^3_{m}(k_{l_{m}})}{J_{m}(k_{l_{m}})+2k_{l_{m}}J'_{m}(k_{l_{m}})}}\\
&\leqslant
2\tau^2\left(\frac{J_{m}(k_{l_{m}}\tau)}{J_{m}(k_{l_{m}})}\right)^2\left(1+2k_{l_{m}}\frac{J'_{m}(k_{l_{m}})}{J_{m}(k_{l_{m}})}\right).
\end{aligned}
\end{equation*}
By virtue of Lemma~\ref{lem:3}, it further holds that
\begin{equation}\label{eq:gg1}
\begin{aligned}
\frac{\|v_m\|^2_{L^2(\Omega_{\tau})}}{\|v_m\|^2_{L^2(\Omega)}}
&\leqslant 2\tau^2\left(\frac{J_{m}(k_{l_{m}}\tau)}{J_{m}(k_{l_{m}})}\right)^2\left(1+2k_{l_{m}}\frac{J'_{m}(k_{l_{m}})}{J_{m}(k_{l_{m}})}\right)\\
&\leqslant
36\mathbf{n}m^4\tau^2\left(\frac{J_{m}(k_{l_{m}}\tau)}{J_{m}(k_{l_{m}})}\right)^2\
\text{.}
\end{aligned}
\end{equation}
Next, $\frac{J_{m}(k_{l_{m}}\tau)}{J_{m}(k_{l_{m}})}$ in \eqref{eq:gg1} can be estimated by the Carlini formula (see formula (30.1) in \cite{Kor02}):
\begin{equation}\label{eq:vm}
\begin{aligned}
\left|\frac{J_{m}(k_{l_{m}}\tau)}{J_{m}(k_{l_{m}})}\right|&=
\left|\frac{\frac{(m (\frac{k_{l_{m}}\tau}{m}))^{m}
\exp \left\{m \sqrt{1-(\frac{k_{l_{m}}\tau}{m})^{2}}\right\}
\exp (-V(m,\frac{k_{l_{m}}\tau}{m}))}{e^{m} \Gamma(m+1)(1-(\frac{k_{l_{m}}\tau}{m})^{2})^{1 / 4}
\left\{1+\sqrt{1-(\frac{k_{l_{m}}\tau}{m})^{2}}\right\}^{m}}}{\frac{(m (\frac{k_{l_{m}}}{m}))^{m}
\exp \left\{m \sqrt{1-(\frac{k_{l_{m}}}{m})^{2}}\right\} \exp (-V(m,\frac{k_{l_{m}}}{m}))}{e^{m}
\Gamma(m+1)(1-(\frac{k_{l_{m}}}{m})^{2})^{1 / 4}\left\{1+\sqrt{1-(\frac{k_{l_{m}}}{m})^{2}}\right\}^{m}}}\right|\\
&=\left|\left(\frac{1-(\frac{k_{l_{m}}}{m})^{2}}{1-(\frac{k_{l_{m}}\tau}{m})^{2}}\right)^{\frac{1}{4}}\times
\left(\frac{\frac{(\frac{k_{l_{m}}\tau}{m})e^{\sqrt{1-(\frac{k_{l_{m}}\tau}{m})^2}}}
{1+\sqrt{1-(\frac{k_{l_{m}}\tau}{m})^2}}}{\frac{(\frac{k_{l_{m}}}{m})e^{\sqrt{1-(\frac{k_{l_{m}}}{m})^2}}}
{1+\sqrt{1-(\frac{k_{l_{m}}}{m})^2}}}\right)^m\times e^{V(m,\frac{k_{l_{m}}}{m})-V(m,\frac{k_{l_{m}}\tau}{m})}\right|\\
&:=I_1\times I_2\times I_3.
\end{aligned}
\end{equation}
Next, we estimate the terms $I_i, i=1,2,3$ in \eqref{eq:vm}. Since
\begin{equation*}
\frac{1}{\mathbf{n}}<\frac{k_{l_{m}}}{m}<\frac{1+\mathbf{n}}{2\mathbf{n}}\ ,
\end{equation*}
we first have
\begin{equation}\label{eq:I1}
I_1:=\left|\left(\frac{1-(\frac{k_{l_{m}}}{m})^{2}}{1-(\frac{k_{l_{m}}\tau}{m})^{2}}\right)^{\frac{1}{4}}\right|<\frac{1}{1-\frac{\mathbf{n}^2+2\mathbf{n}+1}{4\mathbf{n}^2}}<\frac{1}{\mathbf{n}-1}\
\text{.}
\end{equation}
For $I_3$, since the limit of $\frac{k_{l_m}}{m}$ is strictly smaller than $1$, one has
\begin{equation*}
e^{V(m,\frac{k_{l_{m}}}{m})}=1+o(1),\quad
m\rightarrow\infty,
\end{equation*}
and hence
\begin{equation*}
I_3:=\big|e^{V(m,\frac{k_{l_{m}}}{m})-V(m,\frac{k_{l_{m}}\tau}{m})}\big|=\frac{1+o(1)}{1+o(1)}.
\end{equation*}
Therefore there exists a sufficiently large $m_0$ such that when $m>m_0$
\begin{equation}\label{eq:I2}
I_3=|e^{V(m,\frac{k_{l_{m}}}{m})-V(m,\frac{k_{l_{m}}\tau}{m})}|<2.
\end{equation}
For $I_2$, we introduce the auxiliary function $\varphi$:
\begin{equation*}
\varphi(x)=\frac{xe^{\sqrt{1-x^2}}}{1+\sqrt{1-x^2}}.
\end{equation*}
Since $\varphi'(x)=\frac{(1-x^2+\sqrt{1-x^2})e^{\sqrt{1-x^2}}}{(1+\sqrt{1-x^2})^2}>0$, $x\in(0,1)$,  $\varphi$ is monotonically increasing. Noting that
\begin{equation*}
\lim\limits_{m\rightarrow\infty}\frac{k_{l_{m}}}{m}-\frac{k_{l_{m}}\tau}{m}=\frac{(1-\tau)}{\mathbf{n}}>0,
\end{equation*}
there exists $\delta(\tau,\mathbf{n})>0$ such that
\begin{equation}\label{eq:I3}
\frac{\varphi(\frac{k_{l_{m}}}{m})}{\varphi(\frac{k_{l_{m}}\tau}{m})}<1-\delta(\tau,\mathbf{n}).
\end{equation}
Combining (\ref{eq:I1}), (\ref{eq:I2}) and (\ref{eq:I3}), together with straightforward calculations, one can show that
\begin{equation*}
\frac{\left\|v_{m}\right\|_{L^{2}(\Omega_{\tau})}^{2}}{\left\|v_{m}\right\|_{L^{2}(\Omega)}^{2}}
\leqslant
144\frac{\mathbf{n}}{(\mathbf{n}-1)^2}m^4\tau^2(1-\delta(\tau,\mathbf{n}))^{2m}.
\end{equation*}
That is
\begin{equation*}
\lim\limits_{m\rightarrow\infty}\frac{\left\|v_{m}\right\|_{L^{2}(\Omega_{\tau})}^{2}}{\left\|v_{m}\right\|_{L^{2}(\Omega)}^{2}}=0.
\end{equation*}

The proof is complete.
\end{proof}


\begin{theorem}\label{thm:surloc2}
Consider the same setup as that in Theorem~\ref{thm:surloc}. The corresponding transmission eigenfunctions $\{w_m\}_{m\in\mathbb{N}}$ are also surface-localized in the sense that
\begin{equation}\label{eq:hh0}
\lim _{m \rightarrow \infty}
\frac{\left\|w_{m}\right\|_{L^{2}(\Omega_{\tau})}}{\left\|w_{m}\right\|_{L^{2}(\Omega)}}=0.
\end{equation}
\end{theorem}

\begin{proof}
Since $w_m(x)=J_{m}(\mathbf{n}k_{l_{m}}x)$, one has
\begin{equation*}
\begin{aligned}
\mathbf{n}k_{l_{m}}&>\mathbf{n}\frac{j_{m,s_{0}}}{\mathbf{n}}>m\Big(1+2\big(\frac{s_{0}+1}{m}\big)^{\frac{2}{3}}\Big).
\end{aligned}
\end{equation*}
Next, we show that for $m$ sufficiently large, one has $\mathbf{n}k_{l_{m}}\tau<j'_{m,1}$. Indeed, it can be deduced that
\begin{equation}\label{eq:hh1}
\begin{aligned}
\frac{m}{\mathbf{n}k_{l_{m}}}<\frac{j'_{m,1}}{\mathbf{n}k_{l_{m}}}
<\frac{m(1+3(\frac{1}{m})^{\frac{2}{3}}+2(\frac{1}{m})^{\frac{4}{3}})}{m(1+2(\frac{s_{0}+1}{m})^{\frac{2}{3}})},
\end{aligned}
\end{equation}
where we make use of the following fact
\begin{equation*}
\begin{aligned}
\frac{m}{j_{m,s_0+1}}=\frac{m}{\mathbf{n}\cdot\frac{j_{m,s_0+1}}{\mathbf{n}}}<\frac{m}{\mathbf{n}k_{l_{m}}}<\frac{m}{\mathbf{n}\cdot\frac{j_{m,s_0}}{\mathbf{n}}}=\frac{m}{j_{m,s_0}}.
\end{aligned}
\end{equation*}
Since
\begin{equation}\label{eq:hh2}
\begin{aligned}
&\lim\limits_{m\rightarrow\infty}\frac{m}{\mathbf{n}k_{l_{m}}}=1,\\
&\lim\limits_{m\rightarrow\infty}\frac{m(1+3(\frac{1}{m})^{\frac{2}{3}}+2(\frac{1}{m})^{\frac{4}{3}})}{m(1+2(\frac{s_{0}+1}{m})^{\frac{2}{3}})}=1,
\end{aligned}
\end{equation}
we thus have from \eqref{eq:hh1} that
\begin{equation*}
\lim\limits_{m\rightarrow\infty}\frac{j'_{m,1}}{\mathbf{n}k_{l_{m}}}=1.
\end{equation*}
Hence, for $\varepsilon=\frac{1}{2}(1-\tau)$, there exists a $m_0\in\mathbb{N}$ such that when $m>m_0$, we have
\begin{equation*}
\frac{j'_{m,1}}{\mathbf{n}k_{l_{m}}}>1-\varepsilon>\frac 1 2 (1+\tau)>\tau.
\end{equation*}
That is, $\mathbf{n}k_{l_{m}}\tau<j'_{m,1}$.
Next, by using the Carlini formula again, we have
\begin{equation*}
\begin{aligned}\frac{\left\|w_{m}\right\|_{L^{2}(\Omega_{\tau})}^{2}}{\left\|w_{m}\right\|_{L^{2}(\Omega)}^{2}}&=\frac{\int_{0}^{\tau} r J_{m}^{2}(\mathbf{n}k_{l_{m}} r) \mathrm{d} r}{\int_{0}^{1} r J_{m}^{2}(\mathbf{n}k_{l_{m}} r) \mathrm{d} r}\\
&\leqslant\frac{\int_{0}^{\tau} r J_{m}^{2}(\mathbf{n}k_{l_{m}}
r) \mathrm{d}
r}{\int_{0}^{\frac{j'_{m,1}}{\mathbf{n}k_{l_{m}}}} r
J_{m}^{2}(\mathbf{n}k_{l_{m}} r) \mathrm{d} r}.
\end{aligned}
\end{equation*}
The rest of the proof is similar to that of Theorem ~\ref{thm:surloc}, and by straightforward calculations one can show that
\begin{equation*}
\frac{\left\|w_{m}\right\|_{L^{2}(\Omega_{\tau})}^{2}}{\left\|w_{m}\right\|_{L^{2}(\Omega)}^{2}}\leqslant
36 \mathbf{n} m^{4} \tau^{2}\left(\frac{J_{m}(\mathbf{n}k_{l_{m}}
\tau)}{J_{m}(j'_{m,1})}\right)^{2},
\end{equation*}
which readily implies \eqref{eq:hh0}.

The proof is complete.
\end{proof}

\subsection{Three-dimensional result}

The proof of Theorem~\ref{thm:main1} in three dimensions is similar to the two-dimensional case in Theorems~\ref{thm:surloc} and \ref{thm:surloc2}. We only sketch the necessary modifications in what follows.

\begin{theorem}
Consider the same setup as that in Theorem~\ref{thm:main1} in $\mathbb{R}^3$ and assume that $\mathbf{n}>1$ and $\tau\in (0, 1)$ is fixed. Then \eqref{eq:result1} holds true.

\end{theorem}

\begin{proof}
By Fourier expansion, the solutions to \eqref{eq:trans1} in $\mathbb{R}^3$ have the following series expansions:
\begin{equation}\label{eq:ee1}
\begin{aligned}
w(x)&=\sum_{m=0}^{\infty} \sum_{l=-m}^{m} \alpha_{m}^{l} j_{m}(k \mathbf{n}|x|) Y_{m}^{l}(\hat{x}),\\
v(x)&=\sum_{m=0}^{\infty} \sum_{l=-m}^{m} \beta_{m}^{l}
j_{m}(k |x|) Y_{m}^{l}(\hat{x}),
\end{aligned}
\end{equation}
where $\hat x:=x/|x|$, $Y_{m}^{l}$ is the Spherical harmonic function of order $m$ and degree $l$, and
\begin{equation}\label{eq:ee2}
j_{m}(|x|)=\sqrt{\frac{\pi}{2|x|}} J_{m+1 /
2}(|x|),
\end{equation}
is known as the \emph{spherical Bessel function}.
In what follows, we shall look for surface-localized transmission eigenfunctions of the following form:
\begin{equation}\label{eq:3d1}
\begin{aligned}
w_{l,m}(x)&= \alpha_{m} j_{m}(k \mathbf{n}|x|)Y_{m}^{l}(\hat{x}),\\
v_{l,m}(x)&= \beta_{m} j_{m}(k
|x|)Y_{m}^{l}(\hat{x}),
\end{aligned}
\end{equation}
where $\alpha_m$ and $\beta_m$ are constants. By using the two transmission conditions
on $\partial\Omega$, together with straightforward calculations, one can show that
\begin{equation*}
\alpha_{m}=1, \quad
\beta_{m}=\mathbf{n}^{-\frac{1}{2}}\frac{J_{m+\frac{1}{2}}(k
\mathbf{n})}{J_{m+\frac{1}{2}}(k)} \alpha_{m},
\end{equation*}
and $k$ should be a root of the following function
\begin{equation}
f_{m+\frac{1}{2}}(k)=J_{m-\frac{1}{2}}(k)J_{m+\frac{1}{2}}(k\mathbf{n})-\mathbf{n}J_{m+\frac{1}{2}}(k)J_{m-\frac{1}{2}}(k\mathbf{n}).
\end{equation}
Next, we construct the desired transmission eigenvalues by showing that $f_{m+\frac{1}{2}}(k)$ has least one zero point in $\big(\frac{j_{m+\frac{1}{2}, s_{0}}}{\mathbf{n}},
\frac{j_{m+\frac{1}{2}, s_{0}+1}}{\mathbf{n}}\big)$.
Spherical Bessel function satisfies the same estimate in Lemma \ref{lem:2}. Then for any fixed $s_0$,
there exists a sufficiently large $m_{0}(\mathbf{n}, s_{0}) \in \mathbb{N}_{+}$ such that
when $m>m_{0}(\mathbf{n},s_{0})$, we have
\begin{equation*}
\frac{j_{m+\frac{1}{2}, s_{0}+1}}{\mathbf{n}} \leq m+\frac{1}{2}.
\end{equation*}
On the other hand, the spherical Bessel functions possess the following property \cite{LZ18}:
\begin{equation*}
m+\frac{1}{2} \leq j'_{m+\frac{1}{2}, 1}<j_{m+\frac{1}{2},
1}<j'_{m+\frac{1}{2}, 2}<j_{m+\frac{1}{2},
2}<j'_{m+\frac{1}{2},3}<\cdots.
\end{equation*}
Consider the function $f_{m+\frac{1}{2}}(k)$ in $(\frac{j_{m+\frac{1}{2}, s_{0}}}{\mathbf{n}},
\frac{j_{m+\frac{1}{2},s_{0}+1}}{\mathbf{n}})$. By the monotonicity of the Bessel function, we have
\begin{equation*}
J_{m+\frac{1}{2}}(k) \geq 0, \quad k \in\left[0,
j'_{m+\frac{1}{2},1}\right].
\end{equation*}
According to the formula (9.5.2) in \cite{Abr}, we know
\begin{equation*}
\begin{aligned}
&f_{m+\frac{1}{2}}\left(\frac{j_{m+\frac{1}{2}, s_{0}}}{\mathbf{n}}\right)
\cdot f_{m+\frac{1}{2}}\left(\frac{j_{m+\frac{1}{2}, s_{0}+1}}{\mathbf{n}}\right)
\\
=&\, \mathbf{n}^2 J_{m+\frac{1}{2}}\left(\frac{j_{m+\frac{1}{2}, s_{0}}}{\mathbf{n}}\right)
\cdot J_{m+\frac{1}{2}}\left(\frac{j_{m+\frac{1}{2}, s_{0}+1}}{\mathbf{n}}\right)
\cdot J_{m-\frac{1}{2}}(j_{m+\frac{1}{2}, s_{0}}) \cdot J_{m-\frac{1}{2}}(j_{m+\frac{1}{2}, s_{0}+1}) \\
\leq &\, \mathbf{n}^2(J_{m+\frac{1}{2}}(j_{m+\frac{1}{2}, 1}'))^{2} \cdot J_{m-\frac{1}{2}}(j_{m+\frac{1}{2}, s_{0}}) \cdot J_{m-\frac{1}{2}}(j_{m+\frac{1}{2}, s_{0}+1}) \\
<&\, 0.
\end{aligned}
\end{equation*}
Therefore by Rolle's theorem, $f_{m+\frac{1}{2}}(k)$ has at least one zero point in $\left(\frac{j_{m+\frac{1}{2},
s_{0}}}{\mathbf{n}},\frac{j_{m+\frac{1}{2},s_{0}+1}}{\mathbf{n}}\right)$.

For any fixed $s_0\in\mathbb{N}$, we denote the aforementioned zero point in $\left(\frac{j_{m+\frac{1}{2},s_{0}}}{\mathbf{n}},\frac{j_{m+\frac{1}{2},s_{0}+1}}{\mathbf{n}}\right)$ as $k_{l_{m}}$ (comparing to \eqref{eq:te1} in the two-dimensional case). Let $(w_m, v_m)$ be the transmission eigenfunctions associated to the eigenvalue $k_{l_m}$. Next we prove that $\{v_m\}_{m\in\mathbb{N}}$ are surface-localized on $\partial \Omega$. By straightforward calculations, one has
\begin{equation*}
\begin{aligned}
& \left\|v_{m}\right\|_{L^{2}(\Omega_{\tau})}^{2} =\int_{\Omega_{\tau}}\left|j_{m}(k_{l_{m}}|x|)\right|^{2} \mathrm{d} x \\
=& 2 \pi \int_{0}^{\tau}\left|j_{m}(k_{l_{m}} r)\right|^{2} r^2 \mathrm{d} r =2 \pi \int_{0}^{\tau} r^2 j_{m}^{2}(k_{l_{m}} r) \mathrm{d} r\\
=& 2 \pi \int_{0}^{\tau} r^2 \frac{\pi}{2r}J_{m+\frac{1}{2}}^{2}(k_{l_{m}} r) \mathrm{d} r=\pi^2\int_{0}^{\tau} rJ_{m+\frac{1}{2}}^{2}(k_{l_{m}}
r) \mathrm{d} r.
\end{aligned}
\end{equation*}
Hence, it holds that
\begin{equation}\label{eq:tau1}
\frac{\left\|v_{m}\right\|_{L^{2}(\Omega_{\tau})}^{2}}{\left\|v_{m}\right\|_{L^{2}(\Omega)}^{2}}=\frac{\int_{0}^{\tau}
rJ_{m+\frac{1}{2}}^{2}(k_{l_{m}} r) \mathrm{d}
r}{\int_{0}^{1} rJ_{m+\frac{1}{2}}^{2}(k_{l_{m}} r)
\mathrm{d} r}.
\end{equation}
Consider the function $\zeta(r)= rJ_{m+\frac{1}{2}}^{2}(k_{l_{m}} r)$. It can be shown that $\zeta$ is convex in $(0,1)$. Using this fact, one can further estimate that
\begin{equation*}
\begin{aligned}
\frac{\left\|v_{m}\right\|_{L^{2}(\Omega_{\tau})}^{2}}{\left\|v_{m}\right\|_{L^{2}(\Omega)}^{2}}&=\frac{\int_{0}^{\tau} rJ_{m+\frac{1}{2}}^{2}(k_{l_{m}} r) \mathrm{d} r}{\int_{0}^{1} rJ_{m+\frac{1}{2}}^{2}(k_{l_{m}} r) \mathrm{d} r}\\
&\leqslant2 \tau^{2}\left(\frac{J_{m+\frac{1}{2}}(k_{l_{m}}
\tau)}{J_{m+\frac{1}{2}}(k_{l_{m}})}\right)^{2}\left(1+2 k_{l_{m}}
\frac{J_{m+\frac{1}{2}}'(k_{l_{m}})}{J_{m+\frac{1}{2}}(k_{l_{m}})}\right).
\end{aligned}
\end{equation*}
Similar to Lemma~\ref{lem:3}, one can show that there exists constants $C$ and $\gamma$ such that
\begin{equation*}
\frac{k_{l_{m}}J'_{m+\frac{1}{2}}(k_{l_{m}})}{J_{m+\frac{1}{2}}(k_{l_{m}})}<Cm^{\gamma}.
\end{equation*}
Finally, by following a similar argument to the two-dimensional case and combining the above estimates, together with the use of the Carlini formula, one can show that
\begin{equation}\label{eq:kk1}
\lim\limits_{m\rightarrow\infty}\frac{\left\|v_{m}\right\|_{L^{2}(\Omega_{\tau})}}{\left\|v_{m}\right\|_{L^{2}(\Omega)}}=0.
\end{equation}
By following a completely similar argument to that of Theorem~\ref{thm:surloc2} in the two-dimensional case, one can show that \eqref{eq:kk1} also holds for $w_m$.

The proof is complete.
\end{proof}

\section*{Acknowledgement}
The work of Y. Deng was supported by NSF grant of China No. 11971487
and NSF grant of Hunan No. 2020JJ2038. The work of H Liu was
supported by a startup fund from City University of Hong Kong and
the Hong Kong RGC grants (projects 12302018, 12302919 and 12301420).
The work of K. Zhang was supported by the NSF grant of China No.
11871245.

\end{document}